\documentclass{tsp-short}

\usepackage{enumitem}
\usepackage{hyperref}
\usepackage{graphicx}
\usepackage{mathtools}
\usepackage{caption}
\usepackage{placeins}

\newtheorem{theorem}{Theorem}[section]
\newtheorem{lemma}{Lemma}[section]
\newtheorem{corollary}{Corollary}[section]
\newtheorem{proposition}{Proposition}[section]
\theoremstyle{remark}

\theoremstyle{definition}
\newtheorem{definition}{Definition}[section]

\begin{document}

\title
[Polynomial regression  under a mixture of classical and Berkson errors]
{Polynomial regression model under a mixture of classical and Berkson errors}

\author{Oleksandr Liubimov}
\address{Taras Shevchenko National University of Kyiv, Volodymyrska str. 60, 01033 Kyiv, Ukraine}
\curraddr{}
\email{liubimov.o.m@gmail.com}

\author{Alexander Kukush}
\address{Institute of Mathematics NAS of Ukraine, 
 Tereschenkivska str. 3, 01024  Kyiv, Ukraine}
\curraddr{}
\email{akukush@imath.kiev.ua}



\begin{abstract}
A polynomial structural regression model is studied, where the covariate is observed with
a mixture of the classical and Berkson measurement errors. Both variances of the classical and
Berkson errors, as well as some of their higher moments are assumed to be known. Without normality assumptions, consistent estimators of
model parameters are constructed using the Corrected Score method, and conditions for their asymptotic normality are given. Under mild conditions, we found pairs of asymptotically independent estimators. A simulation study illustrates the results. 
\end{abstract}
\maketitle
\section{Introduction and setup}
We deal with a polynomial regression model
$$y = \beta_0 + \beta_1 \xi+\beta_2\xi^2+\ldots+ \beta_d \xi^d+\varepsilon,$$
$$w=x+\delta \quad \text{and} \quad \xi=x+u.$$

\noindent Here, the degree of the polynomial $d\ge1$ is assumed to be known, $w$ is observable surrogate random variable, $y$ is observable response variable, $x$ and $\xi$ are unobservable latent variables, $\delta, u,$ and $\varepsilon$ are centered errors: more precisely, $\delta$ is the classical measurement error, $u$ is Berkson error, and $\varepsilon$ is the error in response; $\beta_0,  \ldots, \beta_d$ are   regression parameters.

The presented 
polynomial model  is an  analogue of the binary measurement error model \cite{Masiuk}, Section 7.2, 
arising in radio-epidemiology. It is the thyroid cancer prevalence model, logistic in flavor. There, an individual exposure dose is measured with a mixture of the classical (instrumental) error and Berkson one. The latter appears because a person from the underlying cohort is  assigned an averaged dose in case of a missing data. 

 Introduce the following assumptions for the polynomial
model.
\begin{enumerate}[label=(\roman*)]
     \item Random variables $x,\delta, u$, and $\varepsilon$ are  independent.
     \item Random variables $\delta, u$, and $\varepsilon$ have  zero mean and $\mathbf{E}[\varepsilon^2+\delta^{2d}+u^{2d}+x^{2d}]<\infty$.
    \item   Moments of $\delta$ and $u$ are known  up to order {2d}. 
    \item The distribution of $x$ is not concentrated at $d$ or fewer points.
\end{enumerate}

\noindent  The regression parameters are unknown, as well as  nuisance parameters $\sigma^2_{\varepsilon}:=\mathbf{Var}(\varepsilon)$ and moments  $\mathbf{E}x^l,  l=1,\ldots,2d.$ 

The model is \textit{structural}, which means that the latent variable $x$ is random, and in the observations we will have independent copies of the model with i.i.d. copies of $x$. The assumptions above are needed to provide the strong consistency of estimators.  For their asymptotic normality we will impose farther assumptions.

Our main goal is to estimate the regression parameters, but we are forced to estimate the nuisance parameters as well - either in order to construct reliable estimates of regression parameters, or to estimate the accuracy of the latter estimates. We study   asym\-ptotic properties of estimators of the
model parameters as the sample size is growing. This way we want to have a better
understanding of the above-mentioned binary model  and similar prevalence models.

In \cite{Yakovliev_Kukush} the underlying model was studied with $d=1$, i.e., it was  a linear regression model. First under the normally distributed $x$ and $u$, the model was rewritten as a linear model with the classical measurement error only, and then the Adjusted Least Squares estimator for regression parameters was derived, which coincides with the Adjusted Naive estimator, the Quasi-Likelihood estimator, and the Maximum Likelihood estimator (the latter in the case where all the underlying random variables were normally distributed) 
\cite{Masiuk}, Section 2.4. Note that assumptions (i) to (iv) are not restrictive, and in case $d=1$ they coincide with the assumptions for the consistency in the linear model \cite{Yakovliev_Kukush}.  

In present paper, the distribution of the latent variable $x$ is not specified. Therefore, the Quasi-Likelihood estimation is not possible, and we use instead the Corrected Score method. It is a \textit{functional} method, which means that it works as well for nonrandom realizations of $x$. In the linear model it is optimal in the sense that it provides the least possible expected loss \cite{Kukush_Maschke}. Of course in the polynomial model if we know the distribution of $x$ up to some nuisance parameters, it is preferable to use the Quasi-Likelihood estimator, since it is more asymptotically efficient compared with the Corrected Score one \cite{Kukush_Malenko}.

The structure of the paper is as follows. In Section 2 we construct  strongly consistent estimators of all the model parameters. Section 3 provides their asymptotic normality with a nonsingular ACM, and we study the asymptotic independence of some pairs of estimators. The latter issue is helpful when constructing the confidence region for the constellation of the model parameters. Section 4 contains the estimation of the proportion in which the error in response and Berkson error affect the precision of the estimate for the leading coefficient $\beta_d,$ the most important regression coefficient characterizing the influence of the regressor $\xi$ on the response variable $y$. A small simulation study for the quadratic model  is reported in Section 5, and Section 6 concludes.
 
We use the following notation.  The symbol $\mathbf{E}$ denotes expectation and acts as
an operator on the total product of quantities, and $\mathbf{Var}$ denotes variance.
Variance of a random variable $\psi$ is denoted as $\sigma_\psi^2,$ e.g., $\sigma_x^2:=\mathbf{Var}(x).$ Abbreviations \textit{r.v., ACM,} and \textit{CS} mean \textit{random variable, asymptotic covariance matrix,} and \textit{Corrected Score} respectively. Equalities for r.v.'s, random vectors, and random matrices hold a.s. Unless otherwise specified, each vector is a column one. Upper index $\mathrm{T}$ denotes transposition, and $A^{-\mathrm{T}}$ is the transposed inverse matrix $A^{-1}$. In the paper, all the vectors are column ones. Convergence in distribution is denoted as $\xrightarrow{d}$.  Bar
means averaging over $i = 1,\ldots,n$, e.g., $\overline{a}:=  n^{-1} \sum_{i=1}^n a_i,  ~\overline{ab}:=n^{-1} \sum_{i=1}^n a_ib_i.$

\section{Consistent estimation of  model parameters}\label{sec:consistent_estimators}

Consider  independent copies of the polynomial model
\begin{equation}\label{eq:regression}
y_i = \beta_0 + \beta_1 \xi_i+\beta_2\xi_i^2+\ldots+ \beta_d \xi_i^d+\varepsilon_i,
\end{equation}
\begin{equation}
\label{eq:errors}
w_i=x_i+\delta_i \quad \text{and} \quad \xi_i=x_i+u_i, \quad i=1,\ldots, n.
\end{equation}
\noindent Based on observations $(w_i, y_i)$, $1\leq i\leq n$, we want to construct  strongly consistent estimators for the unknown parameters.

 We will use the Corrected Score (CS) approach in order to correct the least-squares estimator. In construction of an estimating function we follow the lines of Cheng \& Schneeweiss (1998) \cite{Chen_Schneeweiss_1998} and Kukush \& Tsaregorodtsev (2016) \cite{Kukush_Tsaregorodtsev_2016}, where a polynomial regression with the classical measurement error and without Berkson one was studied.

  Denote 
$\beta=(\beta_0, \beta_1, \ldots, \beta_d)^{\mathrm{T}}$
 and 

$$m^{(2d)}_{x}= \left(\mu_x^{(1)}, \mu_x^{(2)}, \ldots, \mu_x^{(2d)}\right)^\mathrm{T}  = (\mathbf{E}x, \mathbf{E}x^2, \ldots, \mathbf{E}x^{2d})^{\mathrm{T}}.$$

\subsection{Estimation of the moments.} \label{subsec:estimation_of_moments_x}
If both measurement errors  are absent, the correspon\-ding estimating function is
$$S^{(m^{( 2d)}_{x})}\left(x, m^{( 2d)}_{x}\right) = m^{( 2d)}_{x} - (x,x^2, \ldots, x^{2d})^{\mathrm{T}}.$$

When the errors $\delta$ and $u$ are present, according to the CS
approach our goal is to construct the \textit{corrected estimating function} $S_{CS}^{(m^{( 2d)}_{x})}\left(w, m^{( 2d)}_{x}\right)$ with
$$\mathbf{E}\left[S_{CS}^{(m^{( 2d)}_x)}\left(w, m^{( 2d)}_x\right) \, \mid \, x \right] = S_{CS}^{(m^{( 2d)}_{x})}\left(w, m^{( 2d)}_{x}\right).$$

 For this purpose  we construct  the polynomials $t_r(w)$, $0 \leq r \leq 2d$ such that \\$\mathbf{E}[ t_r(w)  \mid x] = x^r$. Since $x$ and $\delta$ are independent, by Newton's binomial formula we get

$$\mathbf{E}[w^r \mid x] = x^r+\sum\limits_{i=0}^{r-1}\binom{r}{i}\mathbf{E}[\delta^{r-i}]\cdot x^i,$$ 

\noindent and thus by setting $t_0(w)=1$, according to Cheng \& Schneeweiss (1998)  \cite{Chen_Schneeweiss_1998} the polynomials $t_r(w), 0\leq r\leq2d$ can be constructed recursively:
$$t_{r}(w)=w^{r}-\sum_{i=0}^{r-1}\binom{r}{i}\mathbf{E}[\delta^{r-i}]\cdot t_i(w).$$

\noindent E.g., $t_0(w)=1$, $t_1(w)=w$, $t_2(w)=w^2-\mathbf{E}\delta^2.$

Therefore, the corrected estimating function is
\begin{equation}\label{estimating_function_moments_of_x}
    S_{CS}^{(m^{( 2d)}_{x})}\left(w, m^{( 2d)}_{x}\right) = m^{( 2d)}_{x} - (t_1(w), t_2(w), \ldots t_{2d}(w))^{\mathrm{T}},
\end{equation}

\noindent and the estimator is given as
\begin{equation}\label{estimator_moments_x}
\widehat{m}^{(2d)}_x=\left(\widehat{\mu}^{(1)}_x, \widehat{\mu}^{(2)}_x, \ldots, \widehat{\mu}^{(2d)}_x\right)^\mathrm{T}=\left(\overline{t_1(w)}, \overline{t_2(w)}, \ldots, \overline{t_{2d}(w)}\right)^{\mathrm{T}}
\end{equation}
\noindent with $~\overline{t_r(w)} = n^{-1}\sum_{1\leq i\leq n} t_r(w_i)$, $1\leq r\leq 2d$.

\subsection{Estimation of the regression coefficients.}
\label{subsec:estimation_regression_coefficients}
If the measurement errors $\delta$ and $u$ are absent, then we  have $w=x=\xi$ and the corresponding least-squares estimating function is given as
$$S^{(\beta)}(x, y, \beta) = [(x^{i+j})_{0\leq i, j \leq d}]~ \beta - (y, yx, yx^2, \ldots, yx^d)^{\mathrm{T}}.$$

\noindent However, when  Berkson error is present we have 
$$y=\beta^{(x)}_0(u)+\beta^{(x)}_1(u) x + \ldots + \beta^{(x)}_d(u)x^{d}+\varepsilon$$
\noindent with $\beta^{(x)}_0(u), \beta_1^{(x)}(u), \ldots, \beta^{(x)}_d(u)$ being  r.v.'s which depend only on $u$. Denote $\beta_r^{(x)}=\mathbf{E}\beta_r^{(x)}(u)$,  $0\leq r \leq d$. Assume for now that the classical error $\delta$ is still absent, then  the least-squares estimating function for the coefficients $\beta^{(x)}:=(\beta_0^{(x)}, \beta_1^{(x)}, \ldots, \beta_d^{(x)})^{\mathrm{T}}$ is 
$$S^{(\beta^{(x)})}(x, y, \beta^{(x)}) = [(x^{i+j})_{0\leq i, j \leq d}]~ \beta^{(x)} - (y, yx, yx^2, \ldots, yx^d)^{\mathrm{T}}.$$

Now, assume that both errors $\delta$ and $u$ are present. The idea is to construct the corrected estimating function for the coefficients $\beta^{(x)}$ and then evaluate $\beta$ via $\beta^{(x)}$ by the link $\beta = A^{-1} \beta^{(x)}$, where $A$ is a nonrandom invertible matrix to be specified later.

We construct the estimating function for $\beta^{(x)}$ in a form
$$S^{(\beta^{(x)})}_{CS}(w,y, \beta^{(x)})=T(w)\beta^{(x)} - h(w,y),$$

\noindent which satisfies the CS identity 
$$\mathbf{E}\left[S^{(\beta^{(x)})}_{CS}(w,y, \beta^{(x)}) \, \mid \, x,y \right] = S^{(\beta^{(x)})}(x,y, \beta^{(x)}).$$

\noindent And then the corrected estimating function for $\beta$ will be given by
$$S^{(\beta)}_{CS}(w,y, \beta) := S^{(\beta^{(x)})}_{CS}(w,y, A\beta)=T(w)A\beta - h(w,y).$$
\noindent For this purpose we  construct the matrix $T(w)$ and the vector $h(w,y)$ with
$$\mathbf{E}[T(w) \, | \, x,y] = (x^{i+j})_{0\leq i,j \leq d} \quad \text{and} \quad \mathbf{E}[h(w,y) \, | \, x,y]=(y,yx, \ldots, yx^d)^{\mathrm{T}}.$$
\vspace{10pt}
The estimation procedure  consists of the following five steps.

\noindent \textit{Step 1. Construct a $(d+1) \times (d+1)$ matrix $T(w)$ such that $\mathbf{E}[T(w)\mid x,y] = (x^{i+j})_{ 0\leq i,j \leq d}$.}

\vspace{5pt}

 Define 
$$T(w)= \left(t_{i+j}(w)\right)_{0\leq i,j \leq d}.$$
\noindent  Note that due to the construction of polynomials $t_0(w), t_1(w),\ldots, t_{2d}(w)$ and independence of $\delta$ and $y$, all the equalities 
$$\mathbf{E}[t_{i+j}(w) \mid x,y]=\mathbf{E}[t_{i+j}(w) \mid x] = x^{i+j}$$ 

\noindent hold, which implies the desired relation.

\noindent \textit{Step 2. In terms of $\beta$ and $u$ find coefficients $$\beta^{(x)}(u) :=(\beta_0^{(x)}(u),\, \beta_{1}^{(x)}(u), \, \ldots, \, \beta_d^{(x)}(u))^{\mathrm{T}}$$ such that $y = \sum\limits_{j=0}^{d}\beta_j^{(x)}(u)\cdot x^j+\varepsilon$}. 

We write $\xi = x+u$ and use Newton's binomial formula together with the independence of $x$ and $u$ to obtain
$$y = \sum_{j=0}^{d} \beta_j \left(\sum_{i=0}^{k}\binom{j}{i}u^{j-i}\cdot x^{i}\right)+\varepsilon.$$
\noindent  Swapping the sums, we get

$$y = \sum_{i=0}^{d} \left(\sum_{j=i}^{d}\binom{j}{i}\beta_j u^{j-i}\right)\cdot x^{i} + \varepsilon.$$

\noindent Therefore, the desired quantities are
$$\beta^{(x)}_i(u) = \sum_{j=i}^{d}\binom{j}{i}\beta_j u^{j-i}.$$

\noindent \textit{Step 3. Construct the vector $h(w,y)$ such that $\mathbf{E}[h(w,y) \mid x,y]=(y, yx, yx^2, \ldots, yx^d)^{\mathrm{T}}$.}

\vspace{5pt}

Note that $\delta$ is independent of $x,\xi,$ and $\varepsilon$, which implies that it is independent of all r.v.'s  $yx^{i}$, $i \geq 0,$ as well. Therefore, by $w=x+\delta$ and Newton's binomial formula we have
$$\mathbf{E}[yw^i \mid x, y] = \sum\limits_{j=0}^i\binom{i}{j}\mathbf{E}[\delta^{i-j}]\cdot yx^j, \quad 0\leq r\leq 2d.$$
\noindent Denote $D=\left(\binom{i}{j}\cdot \mathbf{E}\delta^{i-j}\right)_{0\leq i,j \leq d}$ with  usual convention $\binom{i}{j}= 0$ if $i < j$. Then
$$\mathbf{E}[(yw^0, yw^1, \ldots, yw^d)^{\mathrm{T}} \mid x,y]= D \cdot(yx^0, yx^1, \ldots, yx^{d})^\mathrm{T}.$$

\noindent Therefore, define the vector $h(w,y)$ as
$$h(w,y) = D^{-1}(yw^0, yw^1, \ldots, yw^d)^{\mathrm{T}}=y(t_0(w), t_1(w), \ldots, t_d(w))^{\mathrm{T}}.$$

\vspace{20pt}

\noindent \textit{Step 4. Evaluate $\beta^{(x)}:=\mathbf{E}\,\beta^{(x)}(u)$ in terms of the observed variables.}

By the definition of $\beta^{(x)}$ and independence of $x$ and  $u$ it is easy to verify that

$$[(\mathbf{E}x^{i+j})_{ 0\leq i,j \leq d}] \cdot \beta^{(x)} = \mathbf{E}(y, yx, yx^2, \ldots, yx^d)^{\mathrm{T}}.$$

\noindent Due to the constructions at steps 1 and 3 it is equivalent to
$$\mathbf{E}[T(w)]\cdot \beta^{(x)} = \mathbf{E}\, h(w,y).$$
\noindent Therefore,
$$\beta^{(x)} = (\mathbf{E}\, T(w))^{-1} \, \mathbf{E}\, h(w,y).$$

\vspace{5pt}

\noindent \textit{Step 5. Evaluate $\beta$ in terms of known information.}

\vspace{5pt}

Let $U:=\left(\binom{i}{j}\cdot\mathbf{E}u^{i-j}\right)_{0\leq i,j \leq d}$ with  usual convention $\binom{i}{j}= 0$ if $i < j$. By the formula established at the end of the step 3 and independence of $x,u$ we have
$\beta^{(x)}=U^{\mathrm{T}}\, \beta.$
Therefore, 
$$\beta = U^{-\mathrm{T}} \cdot \beta^{(x)} = U^{-\mathrm{T}} (\mathbf{E}\, T(w))^{-1} \, \mathbf{E}\, h(w,y).$$

The estimator for $\beta$ is given by
\begin{equation}\label{estimator_beta}
    \widehat{\beta} = U^{-\mathrm{T}} \, \overline{T(w)}^{\,-1} \, \overline{h(w,y)},
\end{equation}

\noindent where $\overline{T(w)} = n^{-1}\sum_{i=1}^n T(w_i)$ and $\overline{h(w,y)} = n^{-1}\sum_{i=1}^n h(w_i, y_i)$.
The corresponding corrected estimating function for $\beta$ is as follows
\begin{equation}\label{estimating_function_beta}
    S^{(\beta)}_{CS}(w,y,\beta)= T(w)U^{\mathrm{T}}\beta - h(w,y).
\end{equation}

\subsection{Estimation for  variance of the error in response.}\label{subsec:estimation_sigma_epsilon}
If the measurement errors are absent,  the least-squares estimating function for the variance $\sigma_\varepsilon^2$ is given as
$$S^{(\sigma^2_\varepsilon)} (\xi,y, \beta, \sigma_\varepsilon^2)=\sigma^2_\varepsilon - y^2 +\beta^\mathrm{T}\cdot[(\xi^{i+j})_{0\leq i,j\leq d}\,]\cdot \beta.$$
 This estimating function is unbiased, i.e., for the true values of unknown parameters its expectation equals zero.

In case the errors $\delta$ and $u$ are present, the corrected estimating function should be unbiased as well. Thus, to construct the true estimating function we have to find a matrix $G(w)$ with $\mathbf{E}\, G(w) = (\mathbf{E}\, \xi^{i+j})_{0\leq i,j \leq d}$. It can be done as follows.

First we construct  polynomials $g_r(w)$, $0\leq r \leq 2d,$ such that $\mathbf{E}\, g_r(w) = \mathbf{E} \, \xi^r$. Indeed, due to the independency of $u$ and $x$ we have
$$\mathbf{E}\xi^r = \sum_{j=0}^{r} \binom{r}{j}\mathbf{E}u^{r-j} \cdot \mathbf{E}x^j.$$
 In \hyperref[subsec:estimation_of_moments_x]{Subsection \ref*{subsec:estimation_of_moments_x}} we have constructed the polynomials $t_r(w)$, $0\leq r \leq 2d,$ with $\mathbf{E}t_r(w)=\mathbf{E}x^r$, which allows  to determine the polynomials $g_r(w)$ as
$$g_r(w)=  \sum_{j=0}^{r} \binom{r}{j}\mathbf{E}[u^{r-j}] \cdot t_j(w).$$
\noindent The matrix $G(w)$ is given by

$$G(w)=(g_{i+j}(w))_{0\leq i,j\leq d}.$$
The corrected estimating function for $\sigma^2_{\varepsilon}$ is defined as follows
\begin{equation}\label{estimating_function_sigma_epsilon}
    S^{(\sigma^2_\varepsilon)}_{CS}(w,y, \beta, \sigma_\varepsilon^2)=\sigma^2_\varepsilon - y^2 +\beta^{\mathrm{T}}\, G(w)\, \beta,
\end{equation}

\noindent and the estimator for $\sigma^2_{\varepsilon}$ is 
\begin{equation}\label{sigma_epsilon_estimator}  \widehat{\sigma^2_\varepsilon}=\overline{y^2} - \widehat{\beta}^{\mathrm{T}}\, \overline{G(w)}\, \widehat{\beta}.
\end{equation}

\section{Asymptotic normality and asymptotic independence of estimators}

We deal with independent copies \eqref{eq:regression}-\eqref{eq:errors} of the polynomial  model.
According to  \hyperref[sec:consistent_estimators]{Section \ref*{sec:consistent_estimators}}, the augmented  estimating function $S_{CS}(m^{( 2d)}_{x},\beta, \sigma^2_{\varepsilon}, w, y)$ for the parameters $m^{( 2d)}_{x}, \beta, \sigma^2_{\varepsilon}$ is given by
$$S_{CS}(m^{( 2d)}_{x}, \beta, \sigma^2_{\varepsilon}, w, y) = \left(S^{(m^{( 2d)}_{x})}_{CS},\ S_{CS}^{(\beta)}, \ S_{CS}^{(\sigma^2_{\varepsilon})}\right)^{\mathrm{T}},$$

\noindent where
$$S_{CS}^{(m^{( 2d)}_{x})}\left(w, m^{( 2d)}_{x}\right) = m^{( 2d)}_{x} - (t_1(w), t_2(w), \ldots t_{2d}(w))^{\mathrm{T}},$$
$$S^{(\beta)}_{CS}(w,y,\beta)= T(w)U^{\mathrm{T}}\beta - h(w,y),$$
$$S^{(\sigma^2_\varepsilon)}_{CS}(w,y, \beta, \sigma_\varepsilon^2)=\sigma^2_\varepsilon - y^2 +\beta^{\mathrm{T}}\, G(w)\, \beta.$$
 This estimating function is unbiased, i.e., for the true values of the parameters $m^{( 2d)}_{x},~ \beta$, and $\sigma_{\varepsilon}^2$ we have
$$\mathbf{E} \, S_{CS}(m^{( 2d)}_{x}, \beta, \sigma^2_{\varepsilon}, w, y) = 0.$$

\noindent In addition,  the estimators $\widehat{m}_x^{(2d)}, \widehat{\beta}, \widehat{\sigma}^2_{\varepsilon}$ given in \eqref{estimator_moments_x}, \eqref{estimator_beta}, and \eqref{sigma_epsilon_estimator}  satisfy the estimating equation
$$\frac{1}{n}\sum_{i=1}^{n} S_{CS}(\widehat{m}_x^{(2d)},\widehat{\beta}, \, \widehat{\sigma}^2_{\varepsilon}, \, w_i, \, y_i) = 0.$$

\subsection{Asymptotic normality}

Consider  two additional assumptions.

\begin{enumerate}[resume, label=(\roman*)]
    \item $\mathbf{E}[x^{4d}+ \delta^{4d}+ u^{4d}+ \varepsilon^4] < \infty$.

    \item Support of $\delta$ contains at least $2d$ nonzero points, and support of $\varepsilon$ contains at least $3$ points.
\end{enumerate}

The proof of the next statement is standard.
\begin{lemma}\label{lemma:linear_independence_of_monomials_epsilon_delta}
    Suppose that r.v.'s $\delta$ and $\varepsilon$ are independent and  condition (vi) holds. Then  r.v.'s $\delta, \delta^2, \ldots, \delta^{2d}, \varepsilon, \varepsilon\delta, \varepsilon\delta^2, \ldots, \varepsilon\delta^d$ are linearly independent.
\end{lemma}
\begin{theorem}
    Suppose that the conditions (i)–(v) hold. Then the estimator $$(\widehat{m}_x^{(2d)}, \widehat{\beta},\widehat{\sigma}^2_{\varepsilon})^\mathrm{T}$$ given in \eqref{estimator_moments_x}, \eqref{estimator_beta}, and  \eqref{sigma_epsilon_estimator} is asymptotically normal:

    \begin{equation}\label{eq:asymptotic_normality_of_estimators}
        \sqrt{n} \begin{pmatrix}
        \widehat{m}^{(2d)}_x-m^{(2d)}_x \\
        \widehat{\beta} - \beta \\
        \widehat{\sigma}^2_{\varepsilon} -\sigma^2_{\varepsilon}
    \end{pmatrix} \xrightarrow{d} \mathcal{N}_{3d+2}(0, \Sigma^{(m^{( 2d)}_{x}, \ \beta, \ \sigma^2_{\varepsilon} )}).
    \end{equation}

    \noindent In addition, if the condition (vi) holds, then the ACM~ $~\Sigma^{(m^{( 2d)}_{x}, \ \beta, \ \sigma^2_{\varepsilon} )}$ is nonsingular. 
 
\end{theorem}
\begin{proof}
    We have  
    $$V:= \mathbf{E}~\frac{\partial S}{\partial (m^{( 2d)}_x,\beta, \sigma^2_{\varepsilon})} = \mathbf{E} \ \begin{pmatrix}
        I_{2d} & \mathcal{O} & \mathcal{O} \\
        \mathcal{O} & T(w) \, U^{\mathrm{T}}  & \mathcal{O} \\ 

        \mathcal{O} & 2\beta^{\mathrm{T}}\, G(w) & 1
    \end{pmatrix},$$

    \noindent where $I_{2d}$ is the identity matrix of order $2d$. Clearly, the matrix $V$ is nonsingular since the matrix $\mathbf{E} \, U^{\mathrm{T}}\, T(w)$ is nonsingular. Therefore, by  Theorem A.26 in  \cite{Masiuk}
    we see that \eqref{eq:asymptotic_normality_of_estimators} holds and the asymptotic matrix $\Sigma^{(m^{(, 2d)}_x, \ \beta, \ \sigma^2_{\varepsilon})})$ can be evaluated by the Sandwich Formula

    $$\Sigma^{(m^{(2d)}_x, \ \beta, \ \sigma^2_{\varepsilon} )} = V^{-1} \, C \, V^{-\mathrm{T}}$$

    \noindent with $C:= \mathbf{E}_{(m^{(2d)}_x, \ \beta, \ \sigma^2_{\varepsilon})}\, [S_{CS}S_{CS}^{\mathrm{T}}]$.

    To accomplish  the proof of the theorem it remains to show that matrix $C$ is nonsingular when additionally condition (vi) holds. It suffices to show that components of the random vector $S_{CS}$ are linearly independent for the true values of the parameters $m^{(2d)}_x,~   \beta,$ and $ \sigma^2_{\varepsilon}$.

    First we claim that components of the augmented estimation function
    \begin{equation}\label{eq:augmented}
\left(S_{CS}^{(m^{( 2d)}_x)\mathrm{T}},\ S_{CS}^{(\beta)\mathrm{T}} \right)^\mathrm{T}
\end{equation}

    are linearly independent. It is sufficient to show that  components of the vector
    \begin{equation}\label{eq:conditional}
    \mathbf{E}\left[\left(S_{CS}^{(m^{( 2d)}_x)\mathrm{T}},\ S_{CS}^{(\beta)\mathrm{T}} \right)^\mathrm{T}\mid \delta, \varepsilon\right]
    \end{equation}
    are linearly independent. Notice that according to the constructions given in \hyperref[subsec:estimation_of_moments_x]{Subsection \ref*{subsec:estimation_of_moments_x}} and \hyperref[subsec:estimation_regression_coefficients]{Subsection \ref*{subsec:estimation_regression_coefficients}}, by expanding $w=x+\delta$, $y=\beta_0 + \beta_1 \xi + \ldots + \beta_d \xi^d + \varepsilon$ and applying the fact that $\delta$ is independent of $x$ and $y$, and $\varepsilon$ is independent of $\xi$ and $w$, one can see that the \textit{i}th component of the vector $\mathbf{E}\left[S_{CS}^{(\beta)} \, \mid \, \delta, \varepsilon\right]$ is a bivariate polynomial in $\varepsilon, \delta$ with the leading monomial being exactly $\varepsilon\delta^{i}, i=0, 1,\ldots,d$. In addition, it is clear that the $i$th component of the vector $\mathbf{E}\left[S_{CS}^{(m_x^{(2d)})} \mid \delta, \varepsilon\right]$ is   univariate polynomial in $\delta$ of degree exactly $i$. Therefore, by condition (vi) and \hyperref[lemma:linear_independence_of_monomials_epsilon_delta]{Lemma \ref*{lemma:linear_independence_of_monomials_epsilon_delta}} all the components of  \eqref{eq:conditional}
    are linearly independent, which yields the linear independency of components of \eqref{eq:augmented}. 
    
    Now, we show that the r.v. $$S_{CS}^{(\sigma^2_{\varepsilon})}:=\sigma^2_{\varepsilon}-y^2 + \beta^{\mathrm{T}}\, U\,T(w) U^{\mathrm{T}} \, \beta$$ cannot
    be expressed as a linear combination of the components of  \eqref{eq:augmented}. Notice that because $\varepsilon$ and $w$ are independent,  each component of  $$\mathbf{E}\left[\left(S_{CS}^{(m^{( 2d)}_x)\mathrm{T}},\ S_{CS}^{(\beta)\mathrm{T}} \right)^\mathrm{T}\mid \delta, \varepsilon\right]$$ is a  polynomial in $\varepsilon$ of degree at most $1$. On the other hand, $\mathbf{E}\left[S_{CS}^{(\sigma^2_{\varepsilon})} \, | \, \varepsilon\right]$ is a polynomial in $\varepsilon$ of exact degree $2$. Since by assumption the support of $\varepsilon$ contains at least $3$ points we conclude that $\mathbf{E}\left[S_{CS}^{(\sigma^2_{\varepsilon})} \, | \, \varepsilon\right]$ cannot be a linear combination of components of  $\mathbf{E}\left[\left(S_{CS}^{(m^{( 2d)}_x)\mathrm{T}},\ S_{CS}^{(\beta)\mathrm{T}} \right)^\mathrm{T}\,|\,\varepsilon\right]$. Therefore, $S_{CS}^{(\sigma^2_{\varepsilon})}$ cannot be expressed as a linear combination of components of  \eqref{eq:augmented}.
\end{proof}

     Our proof gives  a  way how to estimate $\Sigma^{(m_x^{(2d)}, \ \beta, \ \sigma^2_{\varepsilon} )}$  consistently.

\begin{proposition}
    Assume that conditions (i)–(v) hold. Then the matrix $\Sigma^{(m_x^{(2d)}, \ \beta, \ \sigma^2_{\varepsilon} )}$ defined in previous theorem can be estimated strongly consistently by

    $$\widehat{\Sigma}^{(m_x^{(2d)}, \ \beta, \ \sigma^2_{\varepsilon})} := \widehat{V}^{-1} \, \widehat{C}  \, \widehat{V}^{-\mathrm{T}},$$

    \noindent where 

    $$\widehat{V} := \frac{1}{n}\sum_{i=1}^{n}\begin{pmatrix}
        I_{2d} & \mathcal{O} & \mathcal{O}\\
        \mathcal{O} & T(w_i) \, U^{\mathrm{T}}  & \mathcal{O} \\ 

        \mathcal{O} & 2\widehat{\beta}^{\mathrm{T}}\, G(w_i) & 1
    \end{pmatrix}, 
~~\widehat{C} := \frac{1}{n}\sum_{i=1}^{n} \widehat{S_i}\, \widehat{S}_i^{\,\mathrm{T}},$$

    \noindent and $~\widehat{S}_i := \left(\widehat{S}_i^{(m_x^{(2d)})\mathrm{T}}, \ \widehat{S}_i^{(\beta)\mathrm{T}}, \ \widehat{S}_i^{(\sigma^2_{\varepsilon})\mathrm{T}}\right)^{\mathrm{T}}$ with

    $$\widehat{S_i}^{(m_x^{(2d)})} := \widehat{m}_x^{(2d)} - (t_1(w_i), t_2(w_i), \ldots, t_{2d}(w_i))^{\mathrm{T}}, $$

    $$\widehat{S}_i^{(\beta)} := T(w_i) \, U^{\mathrm{T}} \widehat{\beta} - h(w_i, y_i),\quad\text{and}\quad\widehat{S}_i^{(\sigma^2_{\varepsilon})} := \widehat{\sigma}^2_{\varepsilon} - y^2_i + \widehat{\beta}^{\, \mathrm{T}} G(w_i) \widehat{\beta}.$$
\end{proposition}

\subsection{Asymptotic independence}

\begin{definition}
    \textit{Let} $\widehat{\theta}=(\widehat{\theta}_1, \ldots, \widehat{\theta}_m)^\mathrm{T} \in \mathbf{R}^m$ \textit{be asymptotically normal estimator of the parameter} $\theta=(\theta_1,  \ldots, \theta_m)^\mathrm{T} \in \mathbf{R}^m$, \textit{i.e.,}
    $$\sqrt{n}(\widehat{\theta}_1-\theta_1,\ldots,
\widehat{\theta}_m -\theta_m)^\mathrm{T}      
   \xrightarrow{d} \mathcal{N}_{m}(0, \Sigma)$$

   \noindent where $\Sigma$ is nonsingular for all possible values of model parameters.   \textit{Then  the estimators $\widehat{\theta}_i$ and  $\widehat{\theta}_j$ are called asymptotically independent if the   matrix $\Sigma$ has zero $(i,j)$-entry} for all possible values of model parameters.
\end{definition}

We introduce a new assumption.

   \begin{enumerate}[resume, label=(\roman*)]
       \item The first $2d$ moments of $\delta$ coincide with those of the distribution $\mathcal{N}(0,\sigma^2_\delta)$. 
   \end{enumerate}

By $\mathbf{He}_k$ denote the $k$th  Hermite polynomial with unit leading coefficient \cite{Szegő}. In addition,  denote by 
$$\mathbf{He}^{[\sigma^2]}_k(z)= \begin{cases}
    \sigma^k \cdot \mathbf{He}_k\left(\frac{z}{\sigma}\right) & \text{for $\sigma^2 > 0$}, \\
    z^k & \text{for $\sigma^2 = 0$},
    
\end{cases}$$ 
\noindent the scaled Hermite polynomials.   The polynomials $\mathbf{He}^{[\sigma^2]}_k(z), k\ge 0,$ are orthogonal in $L_2$ space of r.v.'s under the normal distribution $\mathcal{N}(0,\sigma^2)$ of $z$. 

The next nice property of Hermite polynomials is a direct consequence of Lemma 3.3 \cite{Masiuk}, where the property  was shown for normally distributed $\delta.$
\begin{lemma}\label{lemma:t_are_hermite_polynomials}
    Suppose that  assumptions (i), (ii), and (vii) hold, then for $0\leq r \leq 2d$ it holds
    $$t_r(z) = \mathbf{He}^{[\sigma_\delta^2]}_k(z), z\in \mathbf{R}.$$
\end{lemma}
\begin{lemma}\label{lemma:inner_product_of_t_given_x}
 Suppose that  assumptions (i), (ii) and (viii) hold. Then for all\\ $0\leq a, b\leq 2d$ one has
$$\mathbf{E}\left[t_a(w)t_b(w)\mid x\right] = \sum_{l=0}^{\min\{a,b\}} l! \binom{a}{l}\binom{b}{l}\sigma^{2l}_{\delta} x^{a+b-2l}.$$
\end{lemma}
\begin{proof}
    By \hyperref[lemma:t_are_hermite_polynomials]{Lemma \ref*{lemma:t_are_hermite_polynomials}} and the binomial translation property of Hermite polynomials \cite{Szegő} we have
$$t_m(w)=t_m(x+\delta)=\sum_{j=0}^{m}\binom{m}{j}x^{m-j}t_j(\delta).$$
Therefore,  $$t_a(w)t_r(w)=\left(\sum_{i=0}^{a}\binom{a}{i}x^{a-i}t_i(\delta)\right)\left(\sum_{j=0}^{b}\binom{b}{j}x^{b-j}t_j(\delta)\right).$$

    \noindent Note that due to the orthogonality of Hermite polynomials and assumption (vii) we have

    $$\mathbf{E}~ t_i(\delta)t_j(\delta) = \begin{cases}
        0, & \text{if $i \neq j$,} \\
        i!\sigma_\delta^{2i}, & \text{if $i=j$,}
    \end{cases}$$

    \noindent and by independence of $\delta$ and $x$ we conclude that the desired identity  holds.
\end{proof}

We number the rows and columns of matrix $U^{-\mathrm{T}}H_x^{-1}$ as $0,1,\ldots, d$ in a standard way. Let $R_{i}$ denote the row of $U^{-\mathrm{T}}H_x^{-1}, i=0,1,\ldots, d.$ By $e_k$  denote the $k$th standard basis vector (its coordinates are indexed by numbers $0,1,\ldots, d$). The following proposition provides the algebraic criterion for asymptotic independence of the estimators $\widehat{\beta}_i$ and $\widehat{\mu}_x^{(j)}$.

\begin{proposition}\label{proposition:characterization_of_independence_of_beta_and_mu}
    Suppose that  conditions (i)–(v) and (vii) hold and fix $1\leq i \leq d$,  $1\leq j \leq 2d$. Then the estimators $\widehat{\beta}_i$ and $\widehat{\mu}_x^{(j)}$ are asymptotically independent  if and only if
    $$R_i \cdot \left[\left(\mathbf{E}p^{(j)}_{a,b}(x)\right)_{0\leq a,b \leq d}\right]=0,$$
    \noindent where 
    $$p^{(j)}_{a,b}(x):= \sum_{k=1}^{j} k! \binom{j}{k}\left[\binom{a+b}{k} - \binom{a}{k}\right]\sigma^{2k}_\delta x^{a+b+j-2k}$$

    \noindent with  standard convention $\binom{m}{s} = 0$ for $s > m$.
\end{proposition}

\begin{proof}
     The asymptotic covariance of the estimators $\widehat{\beta}_i$ and $\widehat{\mu_x}^{(j)}$ is 

    \begin{equation*}
        \begin{split}
            \mathbf{Acov}\left[\widehat{\beta}_i, \widehat{\mu}_x^{(j)}\right] & = R_i\cdot \mathbf{E}\left[\left\{T(w)U^\mathrm{T}\beta-y(t_0(w), t_1(w), \ldots, t_d(w))^\mathrm{T}\right\}\cdot (\mu^{(j)}_x - t_j(w))\right] = \\
            & = -R_i \cdot \mathbf{E}\left[\left\{T(w)U^\mathrm{T}\beta-y(t_0(w), t_1(w), \ldots, t_d(w))^\mathrm{T}\right\}\cdot t_j(w)\right] = \\
            & = -R_i \cdot \mathbf{E}\left[t_j(w)\cdot \left\{T(w) \beta_x-(t_0(w), \ldots, t_d(w))^\mathrm{T} (1,x,\ldots, x^d)\beta_x\right\}\right] = \\
            & = -R_i\cdot \mathbf{E}\left[t_j(w)\cdot \left(t_{a+b}(w) - t_a(w)x^b\right)_{0\leq a,b  \leq d}\right] \beta_x= \\ 
            & = -R_i\cdot \mathbf{E}\left[t_j(w)\cdot \left(t_{a+b}(w) - t_a(w)x^b\right)_{0\leq a,b \leq d}\right] U^{\mathrm{T}}\beta.
        \end{split}
    \end{equation*}

    \noindent We see that $\mathbf{Acov}\left[\widehat{\beta}_i, \widehat{\mu}_x^{(j)}\right] = 0$ for all values of $\beta$ if and only if 

    $$R_i\cdot \mathbf{E}\left[t_j(w)\cdot \left(t_{a+b}(w) - t_a(w)x^b\right)_{0\leq a,b \leq d}\right] = 0,$$

    \noindent and due to \hyperref[lemma:inner_product_of_t_given_x]{Lemma \ref*{lemma:inner_product_of_t_given_x}} we have
    \begin{gather*}
        \mathbf{E}\left[t_j(w)\cdot \left(t_{a+b}(w) - t_a(w)x^b\right)\right] = \\
        =\mathbf{E}\left[\sum_{k=0}^{\min\{j, a+b\}}k!\binom{j}{k}\binom{a+b}{k}\sigma_\delta^{2k}x^{a+b+j-2k} - \sum_{k=0}^{\min\{j, a\}}k!\binom{j}{k}\binom{a}{k}\sigma_\delta^{2k}x^{a+b+j-2k}\right] = \\
        = \mathbf{E}\left[\sum_{k=1}^{\min\{j, a+b\}}k!\binom{j}{k}\binom{a+b}{k}\sigma_\delta^{2k}x^{a+b+j-2k} - \sum_{k=1}^{\min\{j, a\}}k!\binom{j}{k}\binom{a}{k}\sigma_\delta^{2k}x^{a+b+j-2k}\right] = \\
        = \mathbf{E}~p^{(j)}_{a,b}(x),
    \end{gather*}
    \noindent which accomplishes the proof. 
\end{proof}
\begin{corollary}\label{corollary:independence_of_mu_x_i_and_beta_j}
        Suppose that  conditions (i)–(v), (vii) hold and in addition  $\delta=0.$  Then for every $1 \leq i \leq d$, $1\leq j \leq 2d,$ the estimators $\widehat{\beta}_i$ and $\widehat{\mu}_x^{(j)}$ are asymptotically independent. 
    \end{corollary}

    \begin{proof}
        The statement follows directly from the \hyperref[proposition:characterization_of_independence_of_beta_and_mu]{Proposition \ref*{proposition:characterization_of_independence_of_beta_and_mu}}, since now $p^{(j)}_{a,b}(x)=0$ for all $1\leq j\leq 2d$.
    \end{proof}

    \begin{theorem}\label{theorem:asymptotic_independence_of_beta_d_and_mu_1_x}
        Suppose that  conditions (i)–(v), (vii) hold and $\sigma^2_\delta > 0$. Then the estimators $\widehat{\beta}_i$ and $\widehat{\mu}_x^{(1)}$ are asymptotically independent  if and only if $i=d$.
    \end{theorem}

    \begin{proof}

    We have  
$$\mathbf{E}p^{(1)}_{a,b}(x)=\mathbf{E}\left[(a+b)\sigma_\delta^{2}\cdot x^{a+b-1} -a\sigma_\delta^{2}\cdot x^{a+b-1}\right] = \sigma^2_\delta\cdot \mathbf{E}[bx^{a+b-1}].$$   By the definition of inverse matrix 
    $$R_i\cdot[ (\mathbf{E}x^{a+b})_{0\leq a,b \leq d}] \cdot U^\mathrm{T} = e^\mathrm{T}_i,$$

    \noindent Thus, using the  polynomials $R_i(x) := R_i \cdot (1,x,\ldots, x^d)^\mathrm{T}$, $1 \leq i\leq d,$ we observe that 
\begin{equation}\label{eq:orthogonality_of_R_and monomials}
        \mathbf{E}[R_i(x)\cdot(1,x,x^2, \ldots, x^d)]=e^\mathrm{T}_iU^{-\mathrm{T}}.
    \end{equation}

    According to  \hyperref[proposition:characterization_of_independence_of_beta_and_mu]{Proposition \ref*{proposition:characterization_of_independence_of_beta_and_mu}} the estimators $\widehat{\beta}_i$ and $\widehat{\mu_x}^{(1)}$, $1\leq i \leq d,$ are asymptotically independent if and only if
    $$R_i\cdot \left(\mathbf{E}\left[bx^{a+b-1}\right]\right)_{0\leq a,b\leq d} = 0,$$

    \noindent which is equivalent to 
$$\mathbf{E}\left(0, R_i(x), \ldots, R_{i}(x)\cdot ix^{i-1}, \ldots, R_{i}(x)\cdot dx^{b-1}\right) = 0.$$
Evidently, due to  \eqref{eq:orthogonality_of_R_and monomials} the latter identity holds if and only if $i=d$.
\end{proof}

\begin{lemma}\label{lemma:expected_value_w_times_xi_polynomials}
    Suppose that  conditions (i), (ii), and (vii) hold true. Then for all $0\leq r \leq2d$ it holds
    $$\mathbf{E}[w\cdot (g_r(w)-\xi^r)]=\sigma^2_\delta\mathbf{E}[r\xi^{r-1}].$$
\end{lemma}

\begin{proof}
    We have $$g_r(w)=\sum_{l=0}^{r}\binom{r}{l}\mathbf{E}[u^{r-l}]t_l(w)$$
    \noindent and
    $$\xi^r = \sum_{l=0}^{r}\binom{r}{l}u^{r-l}x^l.$$

    \noindent Since $w=t_1(w)$,  \hyperref[lemma:inner_product_of_t_given_x]{Lemma \ref*{lemma:inner_product_of_t_given_x}} implies that
    $$\mathbf{E}[wg_r(w)]=\sum_{l=0}^{r}\binom{r}{l}\mathbf{E}[u^{r-l}]\cdot x^{1+l} + \sigma^2_\delta\sum_{l=0}^{r}l\binom{r}{l}\mathbf{E}u^{r-l}\cdot\mathbf{E}x^{l-1}.$$

    \noindent In addition, since $\delta$ has zero   mean and independent of $\xi$ we have
    $$\mathbf{E}w\xi^r=\mathbf{E}x\xi^r=\sum_{l=0}^{r}\binom{r}{l}\mathbf{E}u^{r-l}\cdot\mathbf{E}x^{l+1}.$$
    \noindent Therefore,
    \begin{gather*}
        \mathbf{E}[w(g_r(w)-\xi^r)]=\sigma^2_\delta\sum_{l=0}^{r}l\binom{r}{l}\mathbf{E}u^{r-l}\cdot \mathbf{E}x^{l-1} = \sigma^2_\delta \mathbf{E}\left[\frac{d}{dx}\sum_{l=0}^{r}\binom{r}{l}u^{r-l}x^l\right] = \\
        = \sigma^2_\delta \mathbf{E}\left[\frac{d}{dx}(x+u)^{r}\right] = \sigma^2_\delta r\cdot\mathbf{E}(x+u)^{r-1}\ = \sigma^2_\delta r\cdot\mathbf{E}\xi^{r-1}.
    \end{gather*}

\end{proof}

\begin{theorem}\label{theorem:asymptotic_independence_of_mu_1_x_and_epsilon_variance}
    Suppose that  conditions (i)–(v) and (vii) hold. Then  the estimators $\widehat{\mu}^{(1)}_x$ and $\widehat{\sigma}^2_\varepsilon$ are asymptotically independent.
\end{theorem}

\begin{proof}
Recall that the matrix $U$ was introduced in Subsection 2.2.
    By \\$(\mathrm{P}, 1)=(\mathrm{P_1}, \mathrm{P_2}, \ldots, \mathrm{P}_{d+1}, 1)$ denote the bottom row of the inverse matrix 
$$\begin{pmatrix}
        H_x U^\mathrm{T} & \mathcal{O} \\ 
        2\beta^\mathrm{T}H_\xi & 1
    \end{pmatrix}^{-1} = \begin{pmatrix}
                            U^{-\mathrm{T}}H^{-1}_x & \mathcal{O} \\
                            -2\beta^\mathrm{T}H_\xi U^{-\mathrm{T}}H_x^{-1} & 1
    
\end{pmatrix}$$

\noindent where $H_x = \left(\mathbf{E}[x^{i+j}]\right)_{0\leq i, j \leq d}$  and  $H_\xi = \left(\mathbf{E}[\xi^{i+j}]\right)_{0 \leq i, j \leq d}$ . Thus, $\mathrm{P} = -2\beta^\mathrm{T}H_\xi U^{-\mathrm{T}}H_x^{-1}$.

    Due to the Sandwich Formula the asymptotic covariance between the estimators $\widehat{\mu}^{(1)}_x$ and $\widehat{\sigma}^2_\varepsilon$ is given by the formula
\begin{equation*}
        \begin{split}
            \mathbf{Acov}\left[\widehat{\mu}^{(1)}_x, \widehat{\sigma^2_\varepsilon}\right] & = \mathrm{P}\cdot \mathbf{E}\left[S^{(\beta)}_{CS}\,(\mu^{(1)}_x-w)\right] + \mathbf{E}\left[S^{(\sigma^2_\varepsilon)}_{CS}\, (\mu^{(1)}_x-w)\right] = \\
            & = -\mathbf{E}\left[\mathrm{P} \cdot \left(T(w)U^\mathrm{T} \beta -y(t_0(w), \ldots, t_d(w))^\mathrm{T}\right)w +\right. \\ 
            & \ \ \ \ \ \ \ \ \ \ +\left. \left(\sigma^2_\varepsilon - y^2+\beta^\mathrm{T}G(w)\beta\right)w \right] = \\
            & = -\mathbf{E}\left[\mathrm{P} \cdot w(t_{a+b}(w) - t_a(w)x^b)_{0\leq a, b \leq d} \cdot U^{\mathrm{T}} + \right. \\ 
            & \ \ \ \ \ \ \ \ \ \ + \left. \beta^\mathrm{T} w\left(g_{a+b}(w)-\xi^{a+b}\right)_{0 \leq a, b \leq d}\right]\beta = \\
            & = -\sigma^2_\delta \mathbf{E}\left[\mathrm{P} \cdot (bx^{a+b-1})_{0\leq a, b \leq d}\cdot U^{\mathrm{T}} + \right. 
            \\ & \ \ \ \ \ \ \ \ \ \ + \left. \beta^\mathrm{T} \left((a+b)\xi^{a+b-1}\right)_{0 \leq a, b \leq d}\right]\beta \\
            & = -\sigma^2_\delta \mathbf{E}\left[-2\beta^\mathrm{T}H_\xi U^{-\mathrm{T}}H_x^{-1} \cdot (bx^{a+b-1})_{0\leq a, b \leq d}\cdot U^{\mathrm{T}} + \right. 
            \\ & \ \ \ \ \ \ \ \ \ \ + \left. \beta^\mathrm{T} \left((a+b)\xi^{a+b}\right)_{0 \leq a, b \leq d}\right]\beta = \\
            & = -\sigma^2_\delta \beta^\mathrm{T} \ \mathbf{E}\left[-2H_\xi U^{-\mathrm{T}}H_x^{-1} \cdot (bx^{a+b-1})_{0\leq a, b \leq d}\cdot U^\mathrm{T} + \right. 
            \\ & \ \ \ \ \ \ \ \ \ \ \ \ \ \ \ \ \ \ + \left. \beta^\mathrm{T} \left((a+b)\xi^{a+b-1}\right)_{0 \leq a, b \leq d}\right] \ \beta.
        \end{split}
    \end{equation*}

    \noindent Here, in computations \hyperref[lemma:inner_product_of_t_given_x]{Lemma \ref*{lemma:inner_product_of_t_given_x}} and \hyperref[lemma:expected_value_w_times_xi_polynomials]{Lemma \ref*{lemma:expected_value_w_times_xi_polynomials}} were used.

    Therefore, $\mathbf{Acov}\left[\widehat{\mu}^{(1)}_x, \widehat{\sigma^2_\varepsilon}\right]=0$ holds for all $\beta$ if and only if the following equality holds
    
\begin{equation}\label{eq:defining_expected_value_condition_for_as_indpendence_of_mu_1_var_epsilon}
        \begin{split}
            \mathbf{E}&\left[\left((a+b)\xi^{a+b-1}\right)_{0\leq a,b \leq d} + \right. \\ & + \left.\left((a+b)\xi^{a+b-1}\right)^{\mathrm{T}}_{0\leq a,b \leq d} - \right.\\ 
            & \left.- 2H_\xi U^{-\mathrm{T}}H_x^{-1} \cdot (bx^{a+b-1})_{0\leq a,b \leq d}\cdot U^\mathrm{T} - \right. \\
            & - \left. \left(2H_\xi U^{-\mathrm{T}}H_x^{-1} \cdot (bx^{a+b-1})_{0\leq a,b \leq d}\cdot U^\mathrm{T}\right)^\mathrm{T}\ \right] = 0,
        \end{split}
    \end{equation}

    \noindent which we are going to proveunder the assumptions of the theorem.

    \vspace{10pt}
    
    We have

    \begin{gather*}
        \mathbf{E}\left[(bx^{a+b-1})_{0\leq a, b \leq d}\cdot U^\mathrm{T}\right] = \mathbf{E}\left[(1,x,\ldots, x^{d})^\mathrm{T}\cdot(0, 1, 2x, \ldots, bx^{b-1})\cdot U^\mathrm{T}\right] = \\
        = \mathbf{E}\left[(1,x,\ldots, x^{d})^\mathrm{T}(0, 1, 2\xi, \ldots, b\xi^{b-1})\right]=\mathbf{E}\left[(bx^{a}\xi^{b-1})_{0\leq a,b \leq d}\right].
    \end{gather*}

\noindent Introduce a  matrix $J$,

$$J=\begin{pmatrix}
    0 & 1 & 0 & 0 & \cdots & 0 \\
    0 & 0 & 2 & 0 & \cdots & 0 \\
    0 & 0 & 0 & 3 & \cdots & 0 \\
    \vdots & \vdots & \vdots & \vdots & \ddots & \vdots \\
    0 & 0 & 0 & 0 & \cdots & d \\
    0 & 0 & 0 & 0 & \cdots & 0
\end{pmatrix}$$

\noindent and notice that

$$\mathbf{E}\left[(bx^{a+b-1})_{0\leq a, b \leq d}\cdot U^\mathrm{T}\right] = H_xU^{\mathrm{T}}J.$$

\noindent Therefore, it holds

$$\mathbf{E}\left[H^{-1}_x\cdot (bx^{a+b-1})_{0\leq a, b \leq d}\cdot U^\mathrm{T}\right] = U^{\mathrm{T}}J$$
\noindent and  we finally  deduce that
\begin{gather*}  \mathbf{E}\left[-2H_\xi U^{-\mathrm{T}}H_x^{-1} \cdot (bx^{a+b-1})_{0\leq a, b \leq d}\cdot U^\mathrm{T}\right] = -2H_\xi J  =-2\mathbf{E}\left[\left(b\xi^{a+b-1}\right)_{0\leq a,b \leq d}\right],
\end{gather*}
\begin{gather*}
\mathbf{E}\left[\left(-2H_\xi U^{-\mathrm{T}}H_x^{-1} \cdot (bx^{a+b-1})_{0\leq a, b \leq d}\cdot U^\mathrm{T}\right)^\mathrm{T}\ \right] =-2\mathbf{E}\left[\left(a\xi^{a+b-1}\right)_{0\leq a,b \leq d}\right].
\end{gather*}
\noindent Now it is easy to see that the equality \eqref{eq:defining_expected_value_condition_for_as_indpendence_of_mu_1_var_epsilon} indeed holds, which yields the asymptotic independence of the estimators $\widehat{\mu}^{(1)}_x$ and $\widehat{\sigma}^2_\varepsilon$.
\end{proof}

\textbf{Remark 1.} In the linear case $(d=1)$ \hyperref[theorem:asymptotic_independence_of_beta_d_and_mu_1_x]{Theorem \ref*{theorem:asymptotic_independence_of_beta_d_and_mu_1_x}} and \hyperref[theorem:asymptotic_independence_of_mu_1_x_and_epsilon_variance]{Theorem \ref*{theorem:asymptotic_independence_of_mu_1_x_and_epsilon_variance}} state the asymptotic independence under no restrictions on the distribution of $\delta$ and $x.$ This is a sharper result compared with \cite{Yakovliev_Kukush}, Theorem 3(b), where $\delta$ and $x$ had zero skewness.

\vspace{10pt}

In this subsection we have proven that under the additional condition (vii)  the estimators $\widehat{\mu}_x^{(1)}$ and $\widehat{\beta}_d$ are asymptotically independent as well as the estimators $\widehat{\mu}_x^{(1)}$ and $\widehat{\sigma^2_\varepsilon}$. Alas, one should not expect more pairs of asymptotically independent estimators among the estimators of model parameters.  To illustrate this phenomenon, we compute the ACM explicitly for the quadratic model ($d=2$) with  $\beta_0 = -1$, $\beta_1=2$, $\beta_2=1$, $x \sim \mathcal{N}(1,4)$, $\delta \sim \mathcal{N}(0, 0.25)$, $u \sim \mathcal{N}(0,0.25)$, $\varepsilon \sim \mathcal{N}(0, 1)$. The ACM is shown in the  \hyperref[fig:acm_in_quadratic_berkson]{Table \ref*{fig:acm_in_quadratic_berkson}} .

\begin{figure}[htbp]
    \renewcommand{\figurename}{Table}
    \renewcommand{\thefigure}{1}
    \centering
    \includegraphics[width=0.95\textwidth]{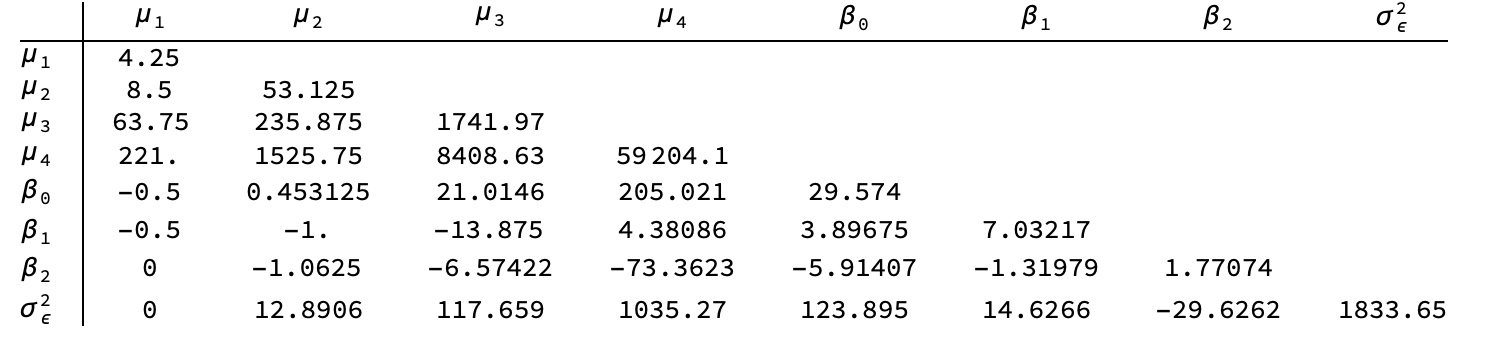}
    \caption{ACM for the quadratic model with the specific parameters}.
    \label{fig:acm_in_quadratic_berkson}

    \addtocounter{figure}{-1}
\end{figure}

By replicating the first half of the proof of  \hyperref[theorem:asymptotic_independence_of_mu_1_x_and_epsilon_variance]{Theorem \ref*{theorem:asymptotic_independence_of_mu_1_x_and_epsilon_variance}}, it is fairly easy to show that if  conditions (i)–(v) hold and $\delta=0$, then for every $1\leq i \leq2d$ the estimators $\widehat{\mu}_x^{(i)}$ and $\widehat{\sigma}^2_\varepsilon$ are asymptotically independent, which together with  \hyperref[corollary:independence_of_mu_x_i_and_beta_j]{Corollary \ref*{corollary:independence_of_mu_x_i_and_beta_j}} shows that it is the classical measurement error $\delta$ that breaks the asymptotic independence, and not the Berkson error $u$.

\FloatBarrier

\section{Proportion in which errors in response and Berkson errors affect the quality of the estimate for the leading coefficient}

Recall the estimating function for $\beta=(\beta_0, \beta_1, \ldots, \beta_d)^\mathrm{T}$  given in \eqref{estimating_function_beta}:  
$$S^{(\beta)}_{CS}=S^{(\beta)}_{CS}(w,y,\beta)= T(w)U^{\mathrm{T}}\beta - h(w,y).$$
\noindent The estimator for $\beta$ given in \eqref{estimator_beta} is asymptotically normal:
$$\sqrt{n}(\widehat{\beta} - \beta) \xrightarrow{d} \mathcal{N}_{d+1}(0, \Sigma^{(\beta)})$$
\noindent with a nonsingular ACM $\Sigma^{(\beta)}$ that can be evaluated by the Sandwich Formula:
$$\Sigma^{(\beta)}=U^{-\mathrm{T}} H_x^{-1} \,\mathbf{E}\left[S^{(\beta)}_{CS} \, S^{(\beta) \, \mathrm{T}}_{CS}\right] H_x^{-1}\, U^{-1}$$
 where $H_x = (\mathbf{E}x^{i+j})_{0\leq i,j \leq d }$. We number rows and columns of each matrix by indices $0,1, \ldots, d$ in a usual way. The asymptotic variance $\sigma^2_{\beta_d}$ of the estimator for $\beta_d$ is the $(d,d)$-entry of the matrix $\Sigma^{(\beta)}$. It is the value we are interested in. 

Since $U^\mathrm{T}$ is an upper-triangular matrix, its inverse $(U^{\mathrm{T}})^{-1}$ is upper-triangular as well  and matrix $U^{-1}$ is lower-triangular. Therefore, the matrices $U^{-\mathrm{T}}=(U^{\mathrm{T}})^{-1}$ and $U^{-1}$ do not affect the $(d,d)$-entry of the matrix $\Sigma^{(\beta)}$. Let $\rho^{(d)}_x$ stand for the $d$th column of the matrix $[(\mathbf{E}x^{i+j})_{0\leq i, j \leq d}]^{-1}$, then $\sigma^2_{\beta_d}$ equals $\rho^{(d) \, \mathrm{T}}_x \cdot \mathbf{E}\left[S^{(\beta)}_{CS} \, S^{(\beta) \, \mathrm{T}}_{CS}\right] \cdot \rho^{(d)}_x$.

 Denote $ (S^{(a,b)})_{0\leq a,b \leq d}=\mathbf{E}~ S^{(\beta)}_{CS} \, S^{(\beta) \mathrm{T}}_{CS} $. We have

\begin{equation*}
    \begin{split}
    S^{(a,b)} & = \mathbf{E} \left(\sum_{i=0}^{d}t_{a+i}(w)\cdot\sum_{j=i}^{d}\binom{j}{i}\mathbf{E}[u^{j-i}]\beta_j - y\,t_a(w)\right) \times \\ & \ \ \ \ \ \times\left(\sum_{i=0}^{d}t_{b+i}(w)\cdot\sum_{j=i}^{d}\binom{j}{i}\mathbf{E}[u^{j-i}]\beta_j - y\,t_b(w)\right)= \\ 
    & = \mathbf{E} \left(\sum_{j=0}^{d}\beta_j\left[\sum_{i=0}^{j}\binom{j}{i}\mathbf{E}[u^{j-i}]t_{a+i}(w)\right] - \sum_{j=0}^{d} \beta_j \xi^{j}t_a(w)-\varepsilon t_a(w)\right) \times \\ & \ \ \ \ \ \times \left(\sum_{j=0}^{d}\beta_j\left[\sum_{i=0}^{j}\binom{j}{i}\mathbf{E}[u^{j-i}]t_{b+i}(w)\right] - \sum_{j=0}^{d} \beta_j \xi^{j}t_b(w)-\varepsilon t_b(w)\right)= \\
    &= \mathbf{E} \left(\sum_{j=0}^{d}\beta_j\left[\sum_{i=0}^{j}\binom{j}{i}\mathbf{E}[u^{j-i}]t_{a+i}(w) - \xi^{j}t_a(w)\right] - \varepsilon t_a(w)\right) \times \\ & \ \ \ \ \ \times  \left(\sum_{j=0}^{d}\beta_j\left[\sum_{i=0}^{j}\binom{j}{i}\mathbf{E}[u^{j-i}]t_{b+i}(w) - \xi^{j}t_b(w)\right] - \varepsilon t_b(w)\right)= \\
    & = \mathbf{E} \left(\sum_{j=0}^{d}\beta_j\left(\sum_{i=0}^{j}\binom{j}{i}\left\{\mathbf{E}[u^{j-i}]t_{a+i}(w) - u^{j-i}\,  x^it_a(w)\right\}\right) - \varepsilon t_a(w)\right) \times \\ & \ \ \ \ \ \times \left(\sum_{j=0}^{d}\beta_j\left(\sum_{i=0}^{j}\binom{j}{i}\left\{\mathbf{E}[u^{j-i}]t_{b+i}(w) - u^{j-i}\,  x^it_b(w)\right\}\right) - \varepsilon t_b(w)\right).
\end{split}
\end{equation*}

\noindent According to the construction of polynomials $t_r$ described in \hyperref[subsec:estimation_of_moments_x]{Subsection \ref*{subsec:estimation_of_moments_x}}, we have
\begin{equation*}
    \begin{split}
        t_{r+i}(w) & =(x+\delta)^{r+i} - \sum_{k=0}^{r+i-1}\binom{r+i}{k}\mathbf{E}[\delta^{r+i-k}]t_k(w) = \\
        & = x^{r+i} + \sum_{k=0}^{r+i-1 }\binom{r+i}{k}\left(\delta^{r+i-k}x^k - \mathbf{E}[\delta^{r+i-k}]t_k(w)\right) =: \\
        & = :x^{r+i}+p_{r+i}(x,\delta).
    \end{split}
\end{equation*}

\noindent Introduce r.v.'s  for $r=0, 1, \ldots, d$:
$$\Upsilon^{(r)}_{\delta} = \sum_{j=0}^{d} \beta_j(t_{r+j}(w)-x^{j}t_r(w)),$$
$$\Upsilon^{(r)}_{u} = \sum_{j=0}^{d}\beta_j\left(\sum_{i=0}^{j} \binom{j}{i}x^{r+i}\{\mathbf{E}[u^{j-i}]-u^{j-i}\}\right),$$
$$\Upsilon^{(r)}_{\delta, u} = \sum_{j=0}^{d}\beta_j\left(\sum_{i=0}^{j-1}\binom{j}{i}\left\{\mathbf{E}[u^{j-i}]p_{r+i}(x,\delta) - u^{j-i}\,  x^i p_r(x, \delta)\right\}\right).$$

\noindent We see that 

\begin{equation*}
    \begin{split}
       S^{(a,b)} & = \mathbf{E}\left(\Upsilon^{(a)}_{\delta} + \Upsilon^{(a)}_u + \Upsilon^{(a)}_{\delta, u} - \varepsilon t_a(w)\right)\left(\Upsilon^{(b)}_{\delta} + \Upsilon^{(b)}_u + \Upsilon^{(b)}_{\delta, u} - \varepsilon t_b(w)\right) \\
       & = \sigma^2_\varepsilon \mathbf{E}t_a(w)t_b(w) + \\ &  \ \ \ \, + \left(\mathbf{E}\Upsilon_\delta^{(a)}\Upsilon_\delta^{(b)}+\mathbf{E}\Upsilon_{u}^{(a)}\Upsilon_{u}^{(b)}+\mathbf{E}\Upsilon^{(a)}_{\delta, u}\Upsilon^{(b)}_{\delta, u} + \mathbf{E}[\Upsilon^{(a)}_{\delta} \Upsilon^{(b)}_{u} + \Upsilon^{(a)}_{u} \Upsilon^{(b)}_{\delta}]\right. + \\ 
        & \ \ \ \, + \left. \ \mathbf{E}[\Upsilon^{(a)}_{\delta, u}(\Upsilon_{\delta}^{(b)} + \Upsilon_{u}^{(b)}) + \Upsilon^{(b)}_{\delta, u}(\Upsilon_{\delta}^{(a)} + \Upsilon_{u}^{(a)})]\right)=:\Delta^{(a,b)}_{\delta, \varepsilon}+\Delta^{(a,b)}_{\delta, u}. 
    \end{split}
\end{equation*}

\noindent Finally, denote 

$$\Lambda^2_{\delta, \varepsilon} = \rho^{(d) \mathrm{T}}_x \left(\Delta^{(a,b)}_{\delta, \varepsilon}\right)_{0\leq a,b \leq d} \ \rho^{(d)}_x,$$

$$\Lambda^2_{\delta, u} = \rho^{(d) \mathrm{T}}_x \left(\Delta^{(a,b)}_{\delta, u}\right)_{0\leq a,b \leq d} \ \rho^{(d)}_x$$

\noindent with nonegative $\Lambda_{\delta, \varepsilon}$ and $\Lambda_{\delta, u}$. We have
$$\sigma^2_{\beta_d} =\Lambda^2_{\delta, \varepsilon} + \Lambda^2_{\delta, u}.$$
It means that $\sqrt{n}(\widehat{\beta}_d - \beta_d)$ can be approximated in distribution by r.v.
$\Lambda_{\delta, \varepsilon} \gamma_1 +  \Lambda_{\delta, u} \gamma_2 $
 with $\gamma_1$ and  $\gamma_2$ independent standard normal. Given the moments of
$x$ and $\sigma^2_\varepsilon,$
 terms $\Lambda_{\delta, \varepsilon} \gamma_1$ and $\Lambda_{\delta, u} \gamma_2$ 
distinguish the influence of the  error in response and Berkson error  on the precision of the  estimator for the leading regression coefficient. Thus,
this influence can be evaluated in proportion $\Lambda_{\delta, \varepsilon}: \Lambda_{\delta, u}.$

 Suppose that $(\beta_1,\beta_2,\ldots,\beta_d)\neq 0$ and   moments of errors $\delta$ and $u$ up to order $2d$ be known.
 Then the influence can be estimated in proportion $\widehat{\Lambda}_{\delta,\varepsilon}:\widehat{\Lambda}_{\delta,u},$ where we plug-in the estimators of the regression parameters and nuisance parameters into the expressions for $\Lambda_{\delta, \varepsilon}$ and 
$\Lambda_{\delta, u}.$
 \\

\section{Simulation}

\begin{figure}[htbp]
    \centering
    \includegraphics[width=0.95\textwidth]{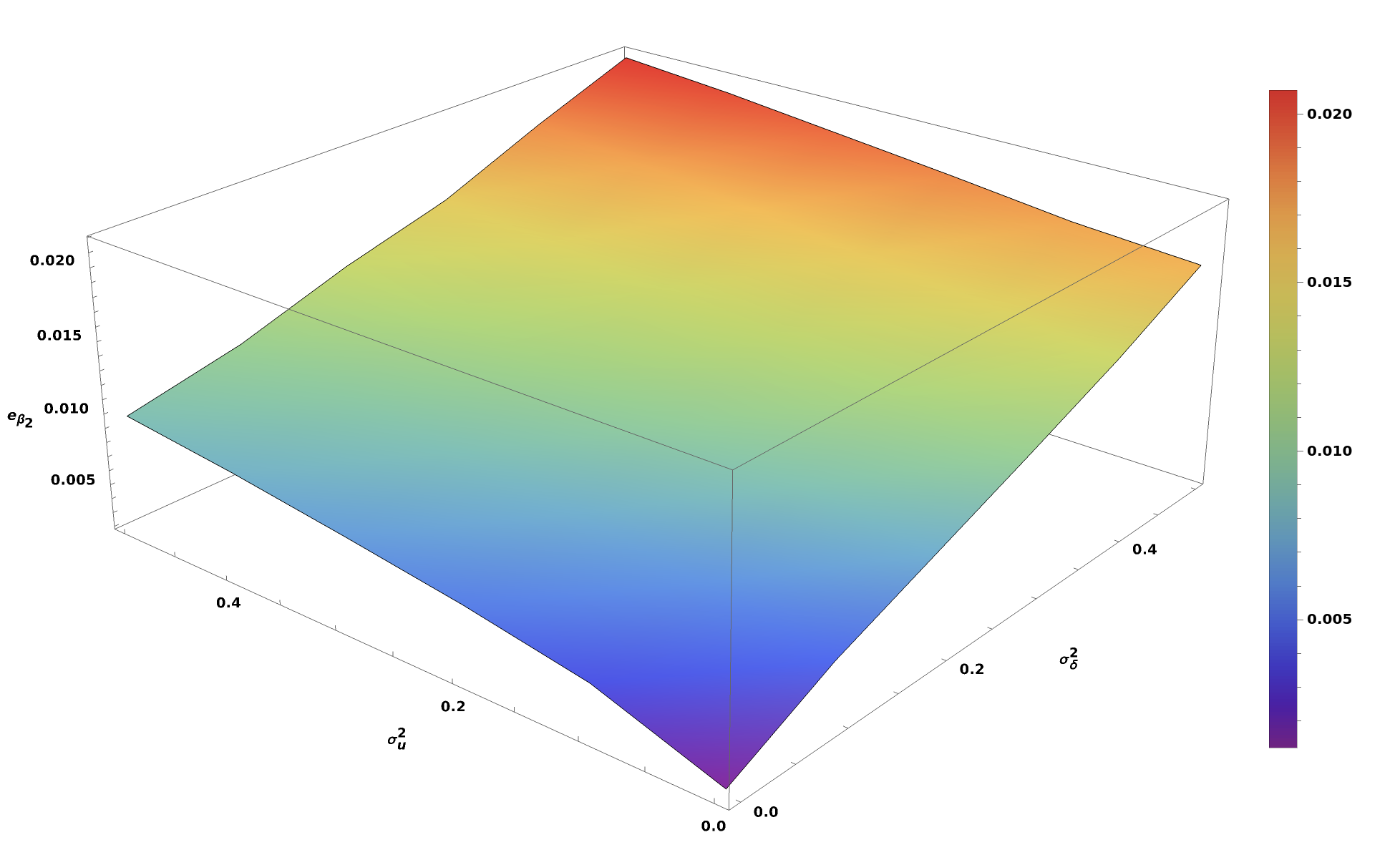}
    \caption{Dependence of $e_{\beta_2}$ on $(\sigma^2_{\delta}, \sigma^2_{u})$}.
    \label{fig:dependence_e_beta_2_on_delta_and_u}
\end{figure}

\begin{figure}[htbp]
    \centering
    \includegraphics[width=0.95\textwidth]{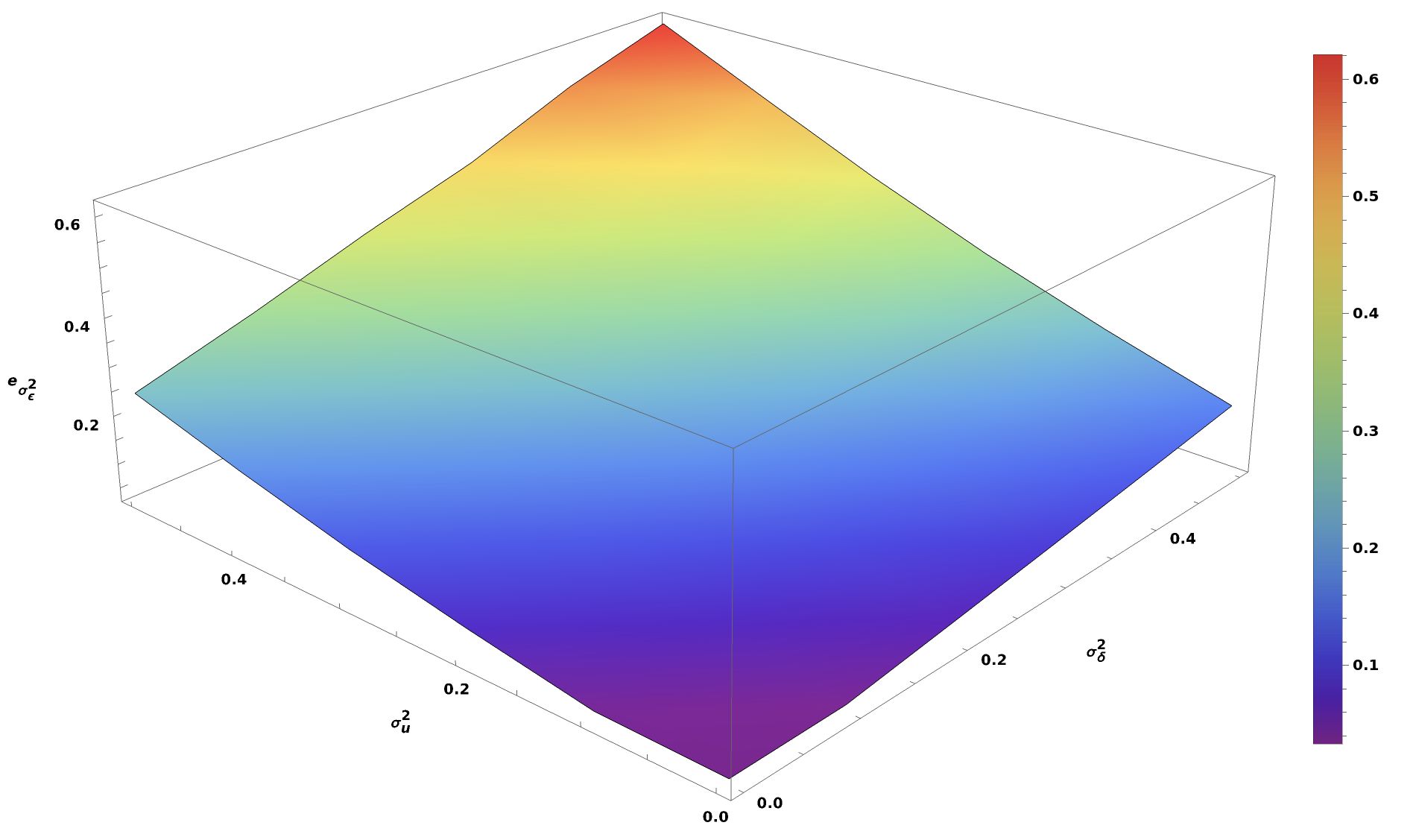}
    \caption{Dependence of $e_{\sigma^2_{\varepsilon}}$ on $(\sigma^2_{\delta}, \sigma^2_{u})$}
    \label{fig:dependence_e_epsilon_variance_on_delta_and_u}
\end{figure}

\begin{figure}[htbp]
    \centering
    \includegraphics[width=0.95\textwidth]{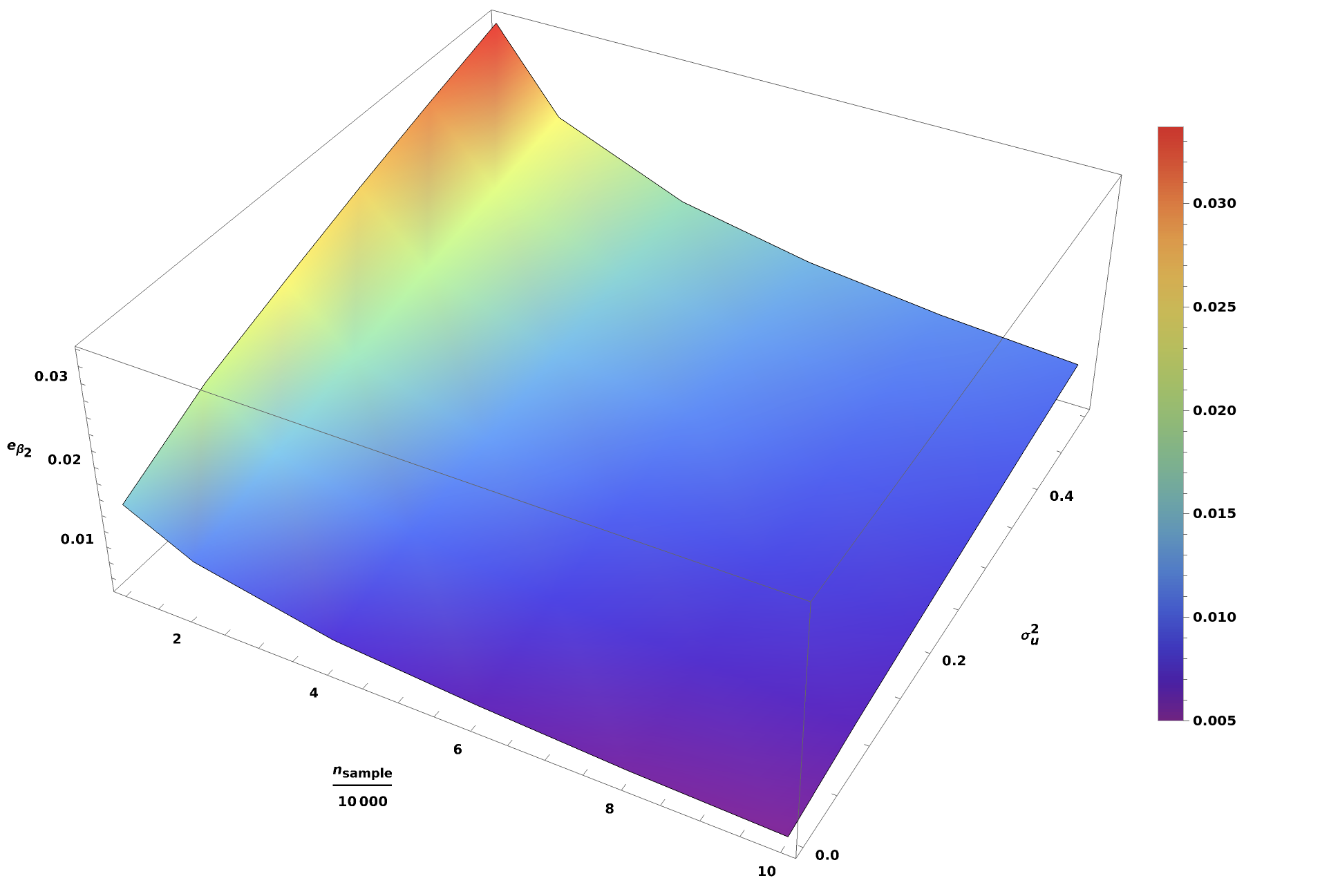}
    \caption{Dependence of $e_{\beta_2}$ on $(n_s, \sigma^2_{u})$}
    \label{fig:dependence_e_beta_2_on_n_and_u}
\end{figure}

\begin{figure}[htbp]
    \centering
    \includegraphics[width=0.95\textwidth]{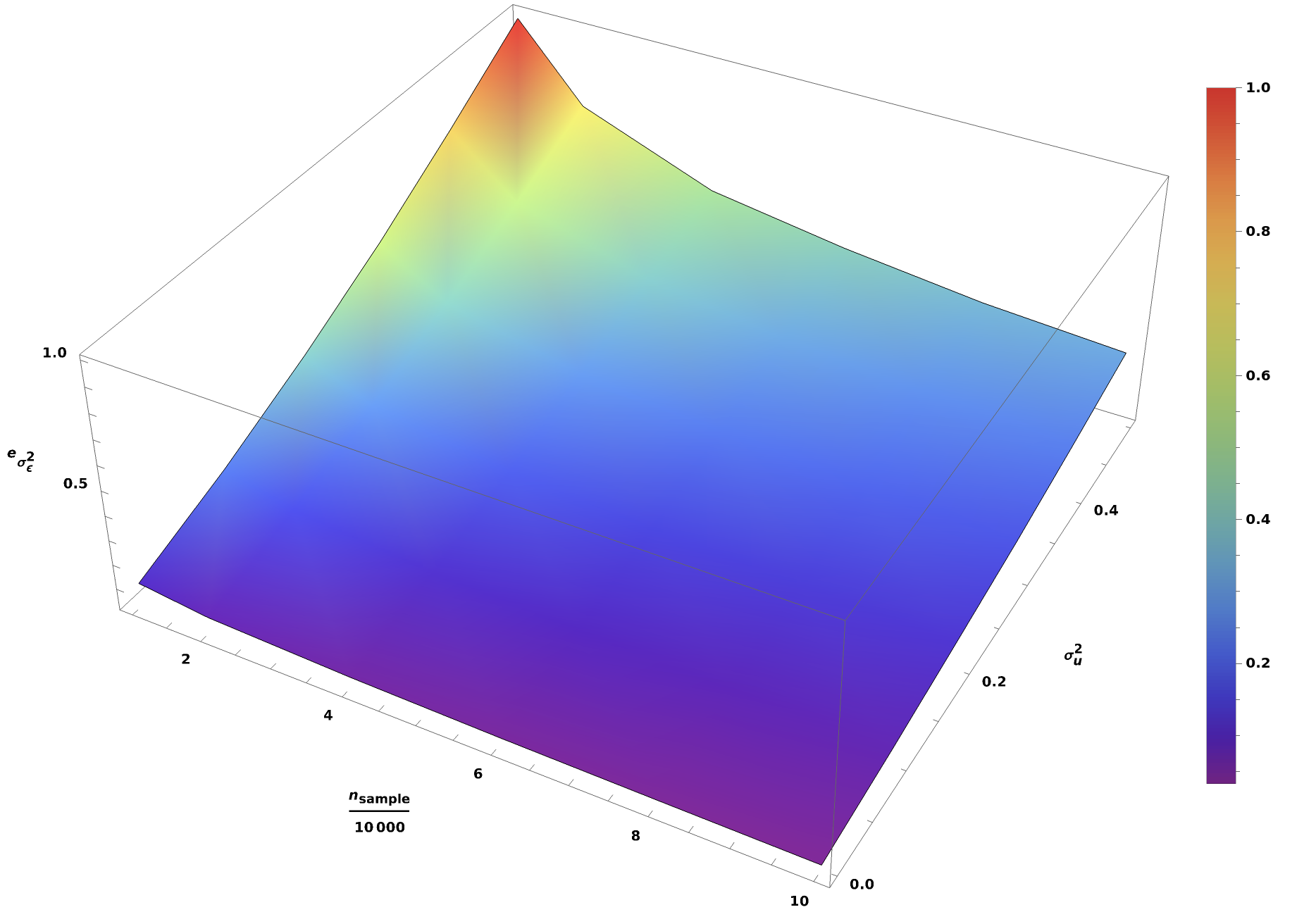}
    \caption{Dependence of $e_{\sigma^2_{\varepsilon}}$ on $(n_{s}, \sigma^2_{u})$}
    \label{fig:dependence_e_epsilon_variance_on_n_and_u}
\end{figure}

\begin{center}
    \begin{figure}[!htb]
   \begin{minipage}{0.48\textwidth}
     \renewcommand{\figurename}{Table}
     \renewcommand{\thefigure}{2}
     \centering
     \includegraphics[width=0.95\linewidth]{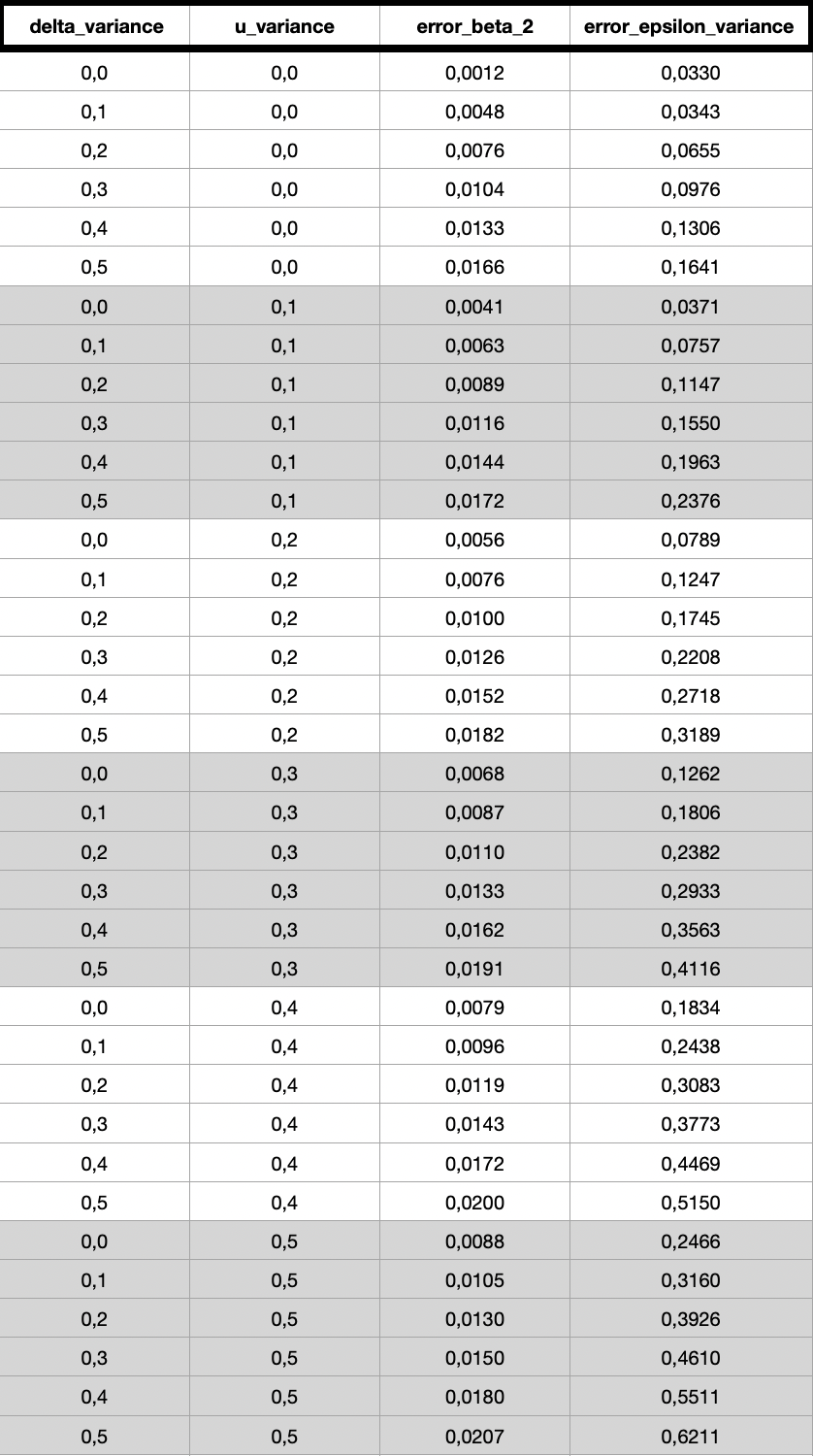}
     \caption{\centering Dependence of $e_{\beta_2}$ and $e_{\sigma^2_\varepsilon}$ on $(\sigma^2_\delta, \sigma^2_u)$}\label{fig:table_dependence_on_delta_variance_and_u_variance}
   \end{minipage}\hfill
   \begin{minipage}{0.48\textwidth}
     \renewcommand{\figurename}{Table}
     \renewcommand{\thefigure}{3}
     \centering
     \includegraphics[width=0.95\linewidth]{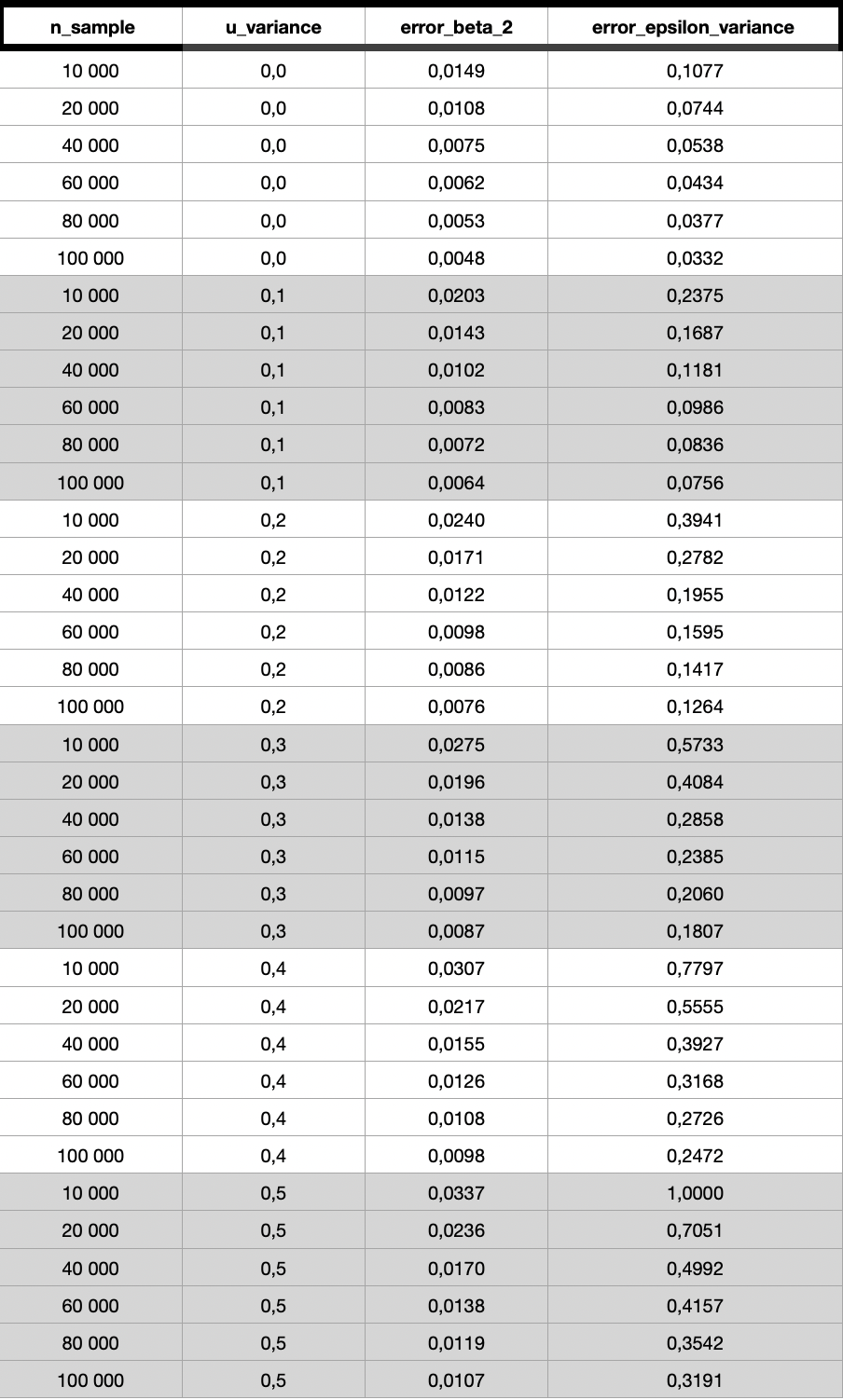}
    \caption{\centering Dependence of $e_{\beta_2}$ and $e_{\sigma^2_\varepsilon}$ on $(n_{s}, \sigma^2_u)$}\label{fig:table_dependence_on_n_sample_and_u_variance}
   \end{minipage}
   \addtocounter{figure}{-2}
\end{figure}
\end{center}

We run the simulation to study the quality of estimates for the coefficient $\beta_2$ and the  variance $\sigma^2_{\varepsilon}$ in the quadratic measurement error model (i.e., for $d=2$). 
In our modeling we have  $\beta_0=-1$, $\beta_1=2$, $\beta_2=1$, $x \sim \mathcal{N}(2,1)$, $\delta \sim \mathcal{N}(0, \sigma^2_{\delta})$, $u \sim \mathcal{N}(0, \sigma^2_{\delta})$, $\varepsilon \sim \mathcal{N}(0, \sigma^2_{\varepsilon})$ with $\sigma^2_{\varepsilon}=0.5$. We run the estimation procedure $N=10^4$ times on the independent samples of size $n_s$ each. To evaluate the quality of estimation we use the (truncated) arithmetic mean of the relative errors:
$$e_{\beta_2} = \underset{1\leq k\leq N}{\mathrm{truncated \ arithmetic \ mean}}\left(\frac{\vert\widehat{\beta}_{2,k}-\beta_2\vert}{|\beta_2|}\right),$$

$$e_{\sigma^2_{\varepsilon}} =\underset{1\leq k\leq N}{\mathrm{truncated \ arithmetic \ mean}}\left(\frac{\vert\widehat{\sigma}^2_{\varepsilon,k}-\sigma^2_{\varepsilon}\vert}{\sigma^2_{\varepsilon}}\right).$$

\noindent Here,  $\widehat{\beta}_{2,,k}$ and  $\widehat{\sigma}^2_{\varepsilon,k}$ stand for the values of the corresponding estimators obtained for the $k$th realization. For the $e_{\beta_2}$ we choose the truncation level to be $10\%$ and for the $e_{\sigma^2_{\varepsilon}}$  $12\%$. The reason of working with the truncated means rather than standard arithmetic means is that the first moments of relative errors do not exist. This is typical for errors-in-variables models \cite{Cheng_Kukush}.

\hyperref[fig:dependence_e_beta_2_on_delta_and_u]{Figure \ref*{fig:dependence_e_beta_2_on_delta_and_u}} and \hyperref[fig:dependence_e_epsilon_variance_on_delta_and_u]{Figure \ref*{fig:dependence_e_epsilon_variance_on_delta_and_u}} demonstrate how $e_{\beta_2}$ and $e_{\sigma^2_{\varepsilon}}$ depend of $(\sigma^2_{\delta}, \sigma^2_{u})$. In each of $N=10^4$ realizations  the sample size $n_s=10^5$ was used. For our next experiment we additionally set $\sigma^2_{\delta} = 0.1$. The aim of this second experiment is to study the dependence of $e_{\beta_2}$ and $e_{\sigma^2_{\varepsilon}}$ on $(n_s, \sigma^2_{u})$. See \hyperref[fig:dependence_e_beta_2_on_n_and_u]{Figure \ref*{fig:dependence_e_beta_2_on_n_and_u}} and \hyperref[fig:dependence_e_epsilon_variance_on_n_and_u]{Figure \ref*{fig:dependence_e_epsilon_variance_on_n_and_u}} below.
The graphs below were constructed based on data, that you may find in \hyperref[fig:table_dependence_on_delta_variance_and_u_variance]{Table \ref*{fig:table_dependence_on_delta_variance_and_u_variance}} and \hyperref[fig:table_dependence_on_n_sample_and_u_variance]{Table \ref*{fig:table_dependence_on_n_sample_and_u_variance}} .

Surprisingly, by observing the \hyperref[fig:dependence_e_beta_2_on_delta_and_u]{Figure \ref*{fig:dependence_e_beta_2_on_delta_and_u}} and \hyperref[fig:dependence_e_epsilon_variance_on_delta_and_u]{Figure \ref*{fig:dependence_e_epsilon_variance_on_delta_and_u}} one can notice that the dependence of $e_{\beta_2}$ and  $e_{\sigma^2_\varepsilon}$ on $(\sigma^2_\delta, \sigma^2_u)$ is almost linear. Looking at \hyperref[fig:dependence_e_beta_2_on_n_and_u]{Figure \ref*{fig:dependence_e_beta_2_on_n_and_u}} and \hyperref[fig:dependence_e_epsilon_variance_on_n_and_u]{Figure \ref*{fig:dependence_e_epsilon_variance_on_n_and_u}} we see that for a fixed $\sigma^2_u$, the dependence of $e_{\beta_2}$ and $e_{\sigma^2_\varepsilon}$ on $n_{s}$ resembles one from the quadratic model in which only the classical measurement error $\delta$ is present. We mention that with growth of  $\sigma^2_u$ the relative errors $e_{\beta_2}$ and $e_{\sigma^2_\varepsilon}$ become more sensitive to the enlargement of the sample size.   

\FloatBarrier

\section{Conclusion}

We studied the polynomial errors-in-variables regression model under a mixture of the classical and Berkson errors. With the aid of Corrected Score method we modified the standard least-squares estimator and constructed the strongly consistent estimators for the following unknown parameters: regression coefficients $\beta_0, \beta_1, \ldots, \beta_d$, variance of the error in response $\sigma^2_{\varepsilon},$ and the first $2d$ moments  of r.v. $x$. We highlight that in the linear case ($d=1$) our estimators coincide with the one obtained in  Yakovlev \& Kukush (2021)~ \cite{Yakovliev_Kukush}, despite the fact that those estimators were constructed in a different way (based on the Least Squares estimator though).

Unlike in the linear model, the estimators for the regression coefficients $\beta_0, \beta_1, \ldots, \beta_d$ in the nonlinear model ($d\geq 2$) require the knowledge of higher moments of Berkson error $u$.
We established the asymptotic normality of the constructed estimators and showed that under fairy general conditions, the corresponding asymptotic covariance matrix (ACM) is nonsingular. In addition we constructed the strongly consistent estimator for the ACM. 

By analyzing the ACM we evaluated the proportion in which, given the moments of $\delta,$ the error in response and Berkson error affect the quality of estimation of the coefficient $\beta_d$. Also we showed that under mild conditions, the estimators for $\beta_d$ and $\mathbf{E}x$ are  asymptotically independent, as well as the estimators for $\mathbf{Var(\varepsilon)}$ and $\mathbf{E}x$. This is helpful to construct a simultaneous confidence region for the pair of parameters.

Finally, we ran the computer simulation and studied the dependence of relative errors of estimation of the parameters $\beta_2$ and $\sigma^2_{\varepsilon}$ in the quadratic model ($d=2$) on the variances of normally distributed classical and Berkson errors, and also studied the dependence of relative errors on the variance of normally distributed Berkson error and the sample size. In particular simulations demonstrate that  estimation of the error variance $\sigma^2_{\varepsilon}$   is computationally more difficult compared with  estimation of the regression coefficients and  it requires notably larger samples to produce  estimates of high quality.

\end{document}